\pdfoutput=1
\documentclass[11pt]{article}

\usepackage[margin=1.1in]{geometry}
\usepackage{amsmath,amssymb,amsthm}
\usepackage{mathtools}
\usepackage{booktabs}
\usepackage{microtype}
\usepackage{xcolor}
\usepackage[colorlinks=true,linkcolor=blue!55!black,citecolor=blue!55!black,urlcolor=blue!55!black]{hyperref}
\usepackage[capitalize,noabbrev]{cleveref}

\theoremstyle{plain}
\newtheorem{lemma}{Lemma}                 
\newtheorem{proposition}{Proposition}
\newtheorem{corollary}{Corollary}
\newtheorem{conjecture}{Conjecture}
\newtheorem*{mainthm}{Main Theorem}
\newtheorem*{thmA}{Theorem A}
\newtheorem*{thmB}{Theorem B}
\theoremstyle{definition}
\newtheorem*{definitionx}{Definition}
\theoremstyle{remark}
\newtheorem{remark}{Remark}

\newcommand{\Z}{\mathbb{Z}}
\newcommand{\C}{\mathbb{C}}
\newcommand{\F}{\mathbb{F}}
\newcommand{\Nmin}{N_{\min}}
\newcommand{\smin}{s_{\min}}
\DeclareMathOperator{\pc}{popcount}
\DeclareMathOperator{\per}{per}

\title{Minimum modulus for the unique multiset-sum problem}
\author{Jos\'e A.\ R.\ Fonollosa\thanks{Universitat Polit\`ecnica de
Catalunya, Barcelona, Spain. \texttt{jose.fonollosa@upc.edu}. ORCID:
\href{https://orcid.org/0000-0001-9513-7939}{0000-0001-9513-7939}.}}
\date{July 2026}

\begin{document}
\maketitle

\begin{abstract}
Fix $n \ge 2$. A set $A = \{a_0 < a_1 < \dots < a_{n-1}\}$ of $n$ residues in
$\Z_N$ is \emph{valid mod $N$} if the all-ones multiset is the \emph{only}
size-$n$ multiset drawn from $A$ whose sum is $p := \sum_i a_i \pmod N$. For the
super-increasing set $A = \{2^k - 1 : 0 \le k \le n-1\}$ we determine the least
valid modulus exactly:
$\Nmin(n) = 2^{\,n} - 2^{\lfloor \log_2 n \rfloor}$ for all $n \ge 2$.
Both directions of the proof are elementary, resting on a sharp
minimal-digit-sum estimate for representations by binary coins, and the full
theorem has been machine-checked in Lean~4/Mathlib for all $n$
(\url{https://github.com/jarfo/min-modulus}). We conjecture
that no size-$n$ residue set admits a smaller valid modulus.

This validity condition is exactly what makes the permanent of an $n \times n$
matrix equal to a single coefficient of a row-product polynomial modulo
$x^N - 1$, extractable by a size-$N$ discrete Fourier (or number-theoretic)
transform; the theorem thus identifies the smallest transform, $N \approx 2^n$,
for which this evaluation is exact. That application --- and the resulting common
framework for the classical formulas of Ryser and Glynn and this transform ---
is developed in a companion paper \cite{fonollosa2026transform}.


\end{abstract}

\section{Introduction}

The permanent of an $n \times n$ matrix,
$\per B = \sum_{\sigma \in S_n} \prod_{i=1}^n b_{i\sigma(i)}$, is a classical
hard object --- computing it exactly is \#P-complete even for $0$--$1$ matrices
\cite{valiant1979}. This paper grew out of a transform-based route to it. A
standard generating-function identity expresses $\per B$ as a single coefficient
of the row-product polynomial
$\prod_i \bigl(\sum_j b_{ij} x^{a_j}\bigr) \bmod (x^N - 1)$, which a size-$N$
discrete Fourier transform (over $\C$, or a number-theoretic transform over a
finite field for exact integer arithmetic) extracts exactly --- \emph{provided}
no size-$n$ multiset of the exponents $a_j$ other than the all-ones one reaches
the target $p = \sum_j a_j$ modulo $N$. Since the cost is governed by $N$, one
wants the \emph{smallest} modulus for which some exponent set has this property.
That transform view of the permanent, and its relation to the classical formulas
of Ryser and Glynn, is developed in the companion paper
\cite{fonollosa2026transform}.

This motivates the following combinatorial question, which we call the
\emph{unique multiset-sum problem} and which seems natural independently of the
application. Call a set $A = \{a_0 < \dots < a_{n-1}\} \subseteq \Z_N$
\emph{valid mod $N$} if the only size-$n$ multiset with elements from $A$
summing to $p = \sum_i a_i \pmod N$ is the multiset containing each element
exactly once. How small can $N$ be?

For the \emph{super-increasing} set
\[
A \;=\; \{\,2^k - 1 : 0 \le k \le n-1\,\} \;=\; \{0, 1, 3, 7, \dots, 2^{n-1}-1\},
\]
we answer the question exactly.

\begin{mainthm}
For every $n \ge 2$, the least modulus at which the super-increasing set is
valid is
\[
\Nmin(n) \;=\; 2^{\,n} - 2^{\lfloor \log_2 n \rfloor}.
\]
\end{mainthm}

The first values, for $n = 2, 3, \dots, 13$, are
\[
2,\; 6,\; 12,\; 28,\; 60,\; 124,\; 248,\; 504,\; 1016,\; 2040,\; 4088,\; 8184.
\]

The proof has two independent halves. The \emph{upper bound}
(\hyperref[sec:validity]{Theorem~A}): the set is valid at $N = 2^n - 2^m$,
$m = \lfloor \log_2 n \rfloor$. After a change of variables, a collision is a
representation of a shifted target by binary coins $\{2^0, \dots, 2^{n-1}\}$
with digit sum exactly $n$; a single step estimate (\cref{lem:step}) shows each
additional multiple of $N$ forces the minimal digit sum up by at least one, and
an induction kills all multiples at once --- no case analysis on $n$, no bound
on the multiple. The \emph{lower bound}
(\hyperref[sec:optimality]{Theorem~B}): every smaller modulus admits an
explicit collision, produced by a complete achievability criterion for the
values of the collision map (\cref{prop:master}). The boundary case is sharp
from both sides in a satisfying way: at gap $2^t$ with $t > m$ a
\emph{negative} multiple of the modulus provides the collision, and at $t = m$
that witness misses by exactly one unit --- the definition of
$m = \lfloor \log_2 n \rfloor$ enters the two halves through the two sides of
the same inequality $2^m \le n \le 2^{m+1}-1$.

Beyond the pencil-and-paper proof, the full statement --- both bounds, all $n$
--- has been formalized and kernel-checked in Lean~4 with Mathlib
\cite{demoura2021,mathlib2020}, with no unproven assumptions
(\cref{sec:certification}).

We conjecture that the super-increasing set is presumably not just convenient but optimal:

\begin{conjecture}\label{conj:global}
For every $n \ge 2$ and every $N < 2^{\,n} - 2^{\lfloor \log_2 n \rfloor}$, no
set of $n$ residues is valid mod $N$; that is, the super-increasing set attains
the least valid modulus over all size-$n$ sets.
\end{conjecture}

\paragraph{Related work.}
The condition studied here is a single-target, multiset, modular relative of
several classical uniqueness conditions on sumsets. Sets all of whose
\emph{subset} sums are distinct (Erd\H{o}s's problem; see
\cite[Problem C8]{guy2004}) and $B_h$/Sidon-type sets, where all $h$-fold sums
are distinct \cite{obryant2004}, both demand far more --- distinctness at
\emph{every} target --- and correspondingly force larger ranges. Validity asks
for uniqueness at the single target $p$ only, which is what allows a modulus
below $2^n$ despite the $\binom{2n-1}{n}$ candidate multisets.
Super-increasing sequences are familiar from knapsack cryptosystems
\cite{merkle1978}, where they make \emph{subset}-sum decoding easy; here the
closely related set $\{2^k - 1\}$ is extremal for a different, multiset
uniqueness property. On the application side, exact permanent algorithms
descend from Ryser \cite{ryser1963} and Glynn \cite{glynn2010}; the companion
paper \cite{fonollosa2026transform} shows that these and the transform used here
are three instances of one framework --- orthogonal evaluation schemes that
isolate the permanent as a single coefficient --- with the present
minimum-modulus result identifying the smallest cyclic instance.

\paragraph{Organization.}
\Cref{sec:problem} fixes notation and states the results.
\Cref{sec:reduction,sec:digitsum} reduce validity to a digit-sum question and
solve the latter. \Cref{sec:validity} proves the upper bound (Theorem~A),
\cref{sec:optimality} the matching lower bound (Theorem~B).
\Cref{sec:certification} reports the machine certification, and
\cref{sec:open} collects open problems.

\section{The problem and the main results}\label{sec:problem}

Fix $n \ge 2$. For a set $A = \{a_0 < a_1 < \dots < a_{n-1}\}$ of $n$ residues
in $\Z_N$, write $p := \sum_{i} a_i \bmod N$. A size-$n$ multiset with elements
from $A$ is encoded by its multiplicity vector
$k = (k_0, \dots, k_{n-1}) \in \Z_{\ge 0}^n$ with $\sum_i k_i = n$.

\begin{definitionx}
$A$ is \emph{valid mod $N$} if the only $k \ge 0$ with $\sum_i k_i = n$ and
$\sum_i k_i a_i \equiv p \pmod N$ is $k = (1, 1, \dots, 1)$.
\end{definitionx}

We study the super-increasing set $A = \{2^k - 1 : 0 \le k \le n-1\}$ and
\[
\Nmin(n) \;=\; \min\{\, N \ge 2 \;:\; A \text{ is valid mod } N \,\}.
\]
Throughout, put
\[
m \;=\; \lfloor \log_2 n \rfloor, \qquad\text{so } 2^m \le n \le 2^{m+1} - 1,
\qquad\text{and}\qquad N \;=\; 2^{\,n} - 2^{\,m}.
\]

The Main Theorem asserts $\Nmin(n) = N$; it splits into Theorem~A
($\Nmin(n) \le N$: the set is valid at $N$; \cref{sec:validity}) and Theorem~B
($\Nmin(n) \ge N$: every smaller modulus admits a collision;
\cref{sec:optimality}).

\section{Reduction to powers of two}\label{sec:reduction}

Substitute $c = k - (1,\dots,1)$, so $c_i \ge -1$ and $\sum_i c_i = 0$. Because
$a_i = 2^i - 1$,
\[
\sum_i k_i a_i - p \;=\; \sum_i c_i a_i \;=\; \sum_i c_i (2^i - 1)
\;=\; \sum_i c_i 2^i - \sum_i c_i \;=\; \sum_i c_i 2^i .
\]
Write $V(c) := \sum_{i=0}^{n-1} c_i 2^i$. Call $c$ \emph{balanced} if
$c_i \ge -1$ for all $i$, $\sum_i c_i = 0$, and $c \ne 0$.

\begin{lemma}[reduction]\label{lem:reduction}
$A$ is valid mod $N$ if and only if no balanced $c$ has $N \mid V(c)$.
\end{lemma}

\begin{proof}
Immediate from the displayed identity: nontrivial solutions $k$ of the validity
congruence correspond bijectively to balanced $c$ with $V(c) \equiv 0 \pmod N$.
\end{proof}

Note that a balanced $c$ has every negative entry equal to $-1$ (since
$c_i \ge -1$), so the distinctness of the residues $a_i \bmod N$ is subsumed:
it is the special case $c = e_i - e_j$, $V = 2^i - 2^j$.

\begin{lemma}[$V \ne 0$]\label{lem:vnonzero}
No balanced $c$ has $V(c) = 0$.
\end{lemma}

\begin{proof}
$V(c) = 0$ means $\sum_i k_i 2^i = \sum_i 2^i = 2^n - 1$ with $\sum_i k_i = n$,
$k_i \ge 0$. By the minimality half of \cref{lem:digitsum} below (whose proof
is independent of this lemma), the digit sum over representations of a fixed
target is minimized by the greedy representation \emph{and by it alone}; for
$2^n - 1$ the greedy representation is the all-ones vector, with digit sum
exactly $n$. So $k = (1,\dots,1)$, i.e.\ $c = 0$.
\end{proof}

\paragraph{Range of $V$, via $k$-space.}
Substituting back $k = c + 1$ (as in \cref{lem:vnonzero}), a balanced $c$ with
value $V$ is the same thing as a vector $k \ge 0$ with $\sum_i k_i = n$ and
\[
\sum_i k_i 2^i \;=\; M \;:=\; V + (2^{\,n} - 1).
\]
Each of the $n$ units of digit sum contributes a coin in $[2^0, 2^{n-1}]$, so
trivially
\begin{equation}\label{eq:range}
n \;\le\; M \;\le\; n \cdot 2^{\,n-1}.
\end{equation}
This yields the two range facts used below.

\begin{itemize}
\item \emph{Upper.} If $V = jN$ with $j \ge 1$, then
$jN = M - (2^n - 1) \le n \cdot 2^{n-1} - (2^n - 1) = (n-2)2^{n-1} + 1
 < (n-1)N$ (using $N \ge 2^n - 2^{n-1} = 2^{n-1}$), so
\begin{equation}\tag{$\star$}\label{eq:star}
1 \;\le\; j \;\le\; n - 2.
\end{equation}
\item \emph{No negative multiples at $N$.} If $N \mid V$, $V \ne 0$
(\cref{lem:vnonzero}) and $V < 0$, then
$M \le (2^n - 1) - N = 2^m - 1 \le n - 1 < n \le M$ --- a contradiction (this
step uses the inequality $2^m \le n$, i.e.\ the definition of $m$, applied to
the trivial bound $n \le M$ of \eqref{eq:range}). Hence $N \mid V$, $V \ne 0$
forces $V = jN$ with $j$ as in \eqref{eq:star}. For $n = 2^m$ the margin is a
single unit: the excluded target $2^m - 1$ falls short of $n$ by exactly one.
\end{itemize}

(Equivalently $V \ge n + 1 - 2^n$, attained by $c = (n-1, -1, \dots, -1)$; the
$M$-form makes the extremal computation unnecessary.)

\section{The digit-sum lemma}\label{sec:digitsum}

For a target $M \ge 0$ and coins $\{2^0, 2^1, \dots, 2^{n-1}\}$, let
$\smin(M)$ denote the minimal digit sum $\sum_i k_i$ over all representations
$k_i \ge 0$, $\sum_i k_i 2^i = M$. Write $\pc(x)$ for the number of ones in the
binary expansion of $x \ge 0$.

\begin{lemma}[digit sum]\label{lem:digitsum}
With $M \bmod 2^{n-1}$ the remainder and $\lfloor \cdot \rfloor$ the quotient,
\[
\smin(M) \;=\; \bigl\lfloor M / 2^{\,n-1} \bigr\rfloor
 \;+\; \pc\!\bigl(M \bmod 2^{\,n-1}\bigr),
\]
and the set of \emph{achievable} digit sums for target $M$ is exactly the
integer interval $[\smin(M),\, M]$.
\end{lemma}

\begin{proof}
\emph{Minimality.} The greedy representation ---
$\lfloor M/2^{n-1} \rfloor$ top coins plus the binary expansion of the
remainder --- attains the stated value. Any representation with some
$k_i \ge 2$, $i \le n-2$, admits a \emph{carry}
$(k_i, k_{i+1}) \to (k_i - 2,\, k_{i+1} + 1)$, lowering the digit sum by $1$;
iterating terminates, and a carry-free representation ($k_i \le 1$ below the
top coin) is determined by its value --- it \emph{is} the greedy one. So every
representation descends to greedy through digit-sum-lowering moves: greedy is
minimal, and it is the unique representation attaining the minimum.

\emph{Contiguity, upward.} Start from greedy, with digit sum $\smin(M)$. While
the current digit sum $s$ is $< M$, some $k_i \ge 1$ with $i \ge 1$ exists ---
otherwise all mass sits at index $0$ and $s = k_0 = M$ --- and the \emph{split}
$(k_i, k_{i-1}) \to (k_i - 1,\, k_{i-1} + 2)$ raises the digit sum by exactly
$1$. Hence the achievable digit sums are exactly $[\smin(M), M]$.
\end{proof}

\begin{remark}
Only the minimality half of \cref{lem:digitsum} is used in
\cref{sec:validity}, and only the contiguity half in \cref{sec:optimality}.
\end{remark}

\section{Validity: the upper bound}\label{sec:validity}

\begin{thmA}
For every $n \ge 2$, the super-increasing set is valid mod
$N = 2^{\,n} - 2^{\,m}$. Equivalently,
$\Nmin(n) \le 2^{\,n} - 2^{\lfloor \log_2 n \rfloor}$.
\end{thmA}

The proof is an induction on the multiple $jN$, powered by a single step
estimate. We state the estimate for a general power-of-two gap, since
\cref{sec:optimality} reuses it at gaps other than $2^m$.

\begin{lemma}[step]\label{lem:step}
Let $0 \le t \le n-1$ and $N_t := 2^{\,n} - 2^{\,t}$. Then for every $M \ge 0$,
\[
\smin(M + N_t) \;\ge\; \smin(M) + 1.
\]
\end{lemma}

\begin{proof}
Write $N_t = 2 \cdot 2^{n-1} - 2^t$ and $M = q \cdot 2^{n-1} + R$ with
$0 \le R < 2^{n-1}$; by \cref{lem:digitsum}, $\smin(M) = q + \pc(R)$. Two cases
on $R$.

\emph{If $R \ge 2^t$:} then $M + N_t = (q+2) \cdot 2^{n-1} + (R - 2^t)$ with
$0 \le R - 2^t < 2^{n-1}$, so
\[
\smin(M + N_t) - \smin(M) \;=\; 2 + \pc(R - 2^t) - \pc(R).
\]
Subtracting a power of two costs at most one bit: if bit $t$ of $R$ is set, the
popcount drops by exactly $1$; if not, the borrow clears the lowest set bit of
$R$ above $t$ --- say bit $u$ --- and sets bits $t, \dots, u-1$, a net change
of $(u - t) - 1 \ge 0$. Either way the difference is $\ge 2 - 1 = 1$.

\emph{If $R < 2^t$:} then
$M + N_t = (q+1) \cdot 2^{n-1} + \bigl(R + (2^{n-1} - 2^t)\bigr)$, and the new
remainder is $< 2^{n-1}$. The mask $2^{n-1} - 2^t$ occupies bits
$t, \dots, n-2$, disjoint from $R < 2^t$, so its $n - 1 - t$ bits add: the
difference is $1 + (n - 1 - t) \ge 1$.
\end{proof}

\begin{proof}[Proof of Theorem A]
Suppose a violation. By \cref{lem:reduction,lem:vnonzero} and the exclusion of
negative multiples at $N$ (\cref{sec:reduction}), some balanced $c$ has
$V(c) = jN$ with $j \ge 1$ (only the lower half of \eqref{eq:star} is needed).
In $k = c + 1$ variables this reads $\sum_i k_i 2^i = M_j$ with
$\sum_i k_i = n$, $k_i \ge 0$, where
\[
M_j \;:=\; jN + (2^{\,n} - 1) \;=\; (j+1)2^{\,n} - j2^{\,m} - 1,
\]
so $\smin(M_j) \le n$ by the definition of $\smin$. But the greedy
representation of $M_0 = 2^n - 1$ is one top coin plus the low ones-block, so
by \cref{lem:digitsum}
\[
\smin(M_0) \;=\; 1 + \pc(2^{\,n-1} - 1) \;=\; 1 + (n-1) \;=\; n,
\]
and \cref{lem:step} with $t = m$ (legitimate: $m \le n-1$) gives, by induction
on $j$,
\[
\smin(M_j) \;\ge\; n + j \;>\; n \qquad \text{for all } j \ge 1
\]
--- a contradiction. No upper bound on $j$ and no small-$n$ case analysis are
needed.
\end{proof}

\begin{remark}[exact surplus]
The induction yields only a bound; the surplus is in fact exactly computable.
For $1 \le j \le 2^{\,n-1-m} - 1$ (a range covering \eqref{eq:star} once
$n \ge 5$) we have $j2^m + 1 \le 2^{n-1}$, so writing
$M_j = 2(j+1)\cdot 2^{n-1} - (j2^m + 1)$,
\[
\bigl\lfloor M_j / 2^{\,n-1} \bigr\rfloor = 2j + 1,
\qquad
M_j \bmod 2^{\,n-1} = \bigl(2^{\,n-1-m} - 1 - j\bigr)\cdot 2^{\,m} + (2^{\,m} - 1).
\]
The two binary blocks (bits $\ge m$ and bits $< m$) do not overlap, and
$\pc(2^{\,n-1-m} - 1 - j) = (n-1-m) - \pc(j)$ (complement within $n-1-m$ bits),
so
\[
\smin(M_j) \;=\; (2j+1) + \bigl(n - 1 - \pc(j)\bigr) \;=\; n + 2j - \pc(j).
\]
Since $\pc(j) \le j$ with equality iff $j \le 1$, the induction bound $n + j$
is tight exactly at $j = 1$.
\end{remark}

\begin{corollary}[slack]\label{cor:slack}
The minimal digit-sum surplus is exactly $\smin(M_j) - n = 2j - \pc(j)$,
minimized at $j = 1$ (surplus $1$). Thus the super-increasing set is valid at
$N = 2^{\,n} - 2^{\,m}$ with exactly one unit of slack, attained at $j = 1$.
\end{corollary}

\section{Optimality: the lower bound}\label{sec:optimality}

We must show every $N'$ with $2 \le N' < 2^n - 2^m$ is \emph{invalid}: some
balanced $c$ has $N' \mid V(c)$. Unlike at $N$ itself --- where the trivial
bound $n \le M$ of \cref{sec:reduction} rules out negative multiples --- a
smaller modulus may (and in one case \emph{must}) be hit by a \emph{negative}
multiple, so we work with signed $V$ throughout.

\begin{proposition}[master achievability criterion]\label{prop:master}
For an integer $V \ne 0$ there is a balanced $c$ with $V(c) = V$ if and only
if, setting $M := V + (2^{\,n} - 1)$,
\[
n \le M \qquad\text{and}\qquad \smin(M) \le n.
\]
\end{proposition}

\begin{proof}
$k = c + 1$ bijects balanced $c$ with value $V$ onto $k \ge 0$,
$\sum_i k_i = n$, $\sum_i k_i 2^i = M$. By \cref{lem:digitsum} the achievable
digit sums for target $M$ are exactly the interval $[\smin(M), M]$, so digit
sum $n$ occurs iff $\smin(M) \le n \le M$.
\end{proof}

Three specializations, computed by quotient/remainder by $2^{n-1}$ as in
\cref{sec:validity} and the complement identity
$\pc(2^{n-1} - 1 - x) = (n-1) - \pc(x)$ for $0 \le x < 2^{n-1}$:
\begin{align*}
0 < V \le 2^{\,n-1}&: & \text{achievable} &\iff \pc(V - 1) \le n - 2,\\
2^{\,n-1} < V < 2^{\,n}&: & \text{achievable} &\iff \pc(V - 2^{\,n-1} - 1) \le n - 3,\\
V < 0&: & \text{achievable} &\iff n \le M \quad (\text{i.e.\ } V \ge n + 1 - 2^{\,n}).
\end{align*}
(For $V < 0$: $M < 2^n - 1$, so $\lfloor M/2^{n-1} \rfloor \le 1$ and
$\smin(M) \le 1 + (n-1) = n$ automatically --- only the trivial range bound
$n \le M$ bites: every negative value down to that floor is achievable.)

\begin{thmB}
For every $n \ge 2$ and every $2 \le N' < 2^{\,n} - 2^{\,m}$, the
super-increasing set is invalid mod $N'$. Hence
$\Nmin(n) = 2^{\,n} - 2^{\lfloor \log_2 n \rfloor}$.
\end{thmB}

\begin{proof}
For $n = 2$ the range is empty; assume $n \ge 3$. Four cases.

\emph{(i) $2 \le N' < 2^{n-1}$.} Take $V = N'$. Then $V - 1 < 2^{n-1} - 1$, so
its $n-1$ low bits are not all ones: $\pc(V - 1) \le n - 2$. Invalid at
$j = 1$.

\emph{(ii) $N' = 2^{n-1}$.} Positive multiples all fail: $V = j \cdot 2^{n-1}$
gives $M = (j+2)2^{n-1} - 1$, so $\smin(M) = (j+1) + (n-1) = n + j > n$. Take
$V = -2^{n-1}$: achievable iff $2^{n-1} \ge n + 1$, true for $n \ge 3$
(equality at $n = 3$). Invalid, necessarily by a negative multiple. (E.g.\
$n = 4$, $N' = 8$: $k = (3,0,1,0)$, i.e.\ the multiset $\{0,0,0,3\}$ collides
with $\{0,1,3,7\}$ --- both sum to $3$ mod $8$.)

\emph{(iii) $2^{n-1} < N' < 2^n - 2^m$ with $s := 2^n - N'$ not a power of
two.} Here $2^m < s < 2^{n-1}$ and $\pc(s) \ge 2$. Take $V = N'$: then
$V - 2^{n-1} - 1 = (2^{n-1} - 1) - s$, the complement of $s$ in $n-1$ bits, so
$\pc = (n-1) - \pc(s) \le n - 3$. Invalid at $j = 1$. (In particular, every gap
$s$ with $\pc(s) \ge 2$ is already invalid at $j = 1$.)

\emph{(iv) $N' = 2^n - 2^t$ a single-power gap, $m + 1 \le t \le n - 2$.} By
the case-(iii) computation, $j = 1$ now gives $\pc = n - 2 > n - 3$:
impossible. Instead take
\[
V \;=\; -N' \;=\; 2^{\,t} - 2^{\,n}, \qquad \text{i.e.} \quad M = 2^{\,t} - 1:
\]
$\smin(M) = t \le n - 2 < n$, and the range condition $n \le M = 2^t - 1$ holds
because
\[
2^{\,t} \;\ge\; 2^{\,m+1} \;\ge\; n + 1,
\]
by the definition $m = \lfloor \log_2 n \rfloor$ (i.e.\ $n \le 2^{m+1} - 1$).
Invalid at $j = -1$.
\end{proof}

In summary (writing $s := 2^n - N'$):

\begin{center}
\begin{tabular}{@{}llll@{}}
\toprule
case & range & witness target $M$ & mechanism \\
\midrule
(i)   & $2 \le N' < 2^{n-1}$                    & $2^n - 1 + N'$ & $V = N'$ ($j = 1$) \\
(ii)  & $N' = 2^{n-1}$                          & $2^{n-1} - 1$  & $V = -N'$ --- negative, necessarily \\
(iii) & $2^{n-1} < N'$, $\pc(s) \ge 2$          & $2^n - 1 + N'$ & $V = N'$ ($j = 1$) \\
(iv)  & $s = 2^t$, $m < t \le n - 2$            & $2^t - 1$      & $V = -N'$ ($j = -1$) \\
\bottomrule
\end{tabular}
\end{center}

\paragraph{Sharpness of the bound.}
The boundary is exactly $t = m$, from both sides:
\begin{itemize}
\item $t \ge m + 1$ $\implies$ $2^t - 1 \ge n$: the $n$ units of digit sum fit
below bit $t$, so the witness $V = 2^t - 2^n$ exists and $2^n - 2^t$ is
invalid.
\item $t = m$ $\implies$ $2^t - 1 \le n - 1 < n$: the same attempt falls below
the floor $n + 1 - 2^n$ (by margin $1$ when $n = 2^m$), and Theorem~A rules out
every positive multiple too: $2^n - 2^m$ is valid.
\end{itemize}

\begin{remark}
In case (iv) the witness is necessarily negative ($j = -1$): no positive
multiple yields a witness at a single-power gap. Indeed \cref{lem:step} applies
at the gap $2^t$ just as at $2^m$, so the induction of \cref{sec:validity}
gives $\smin(jN' + 2^n - 1) \ge n + j > n$ for \emph{every} $j \ge 1$. A
concrete $j = -1$ witness for $n = 5$, $N' = 24 = 2^5 - 2^3$: $k = (3,2,0,0,0)$,
i.e.\ $\{0,0,0,1,1\}$ collides with $\{0,1,3,7,15\}$ --- both sum to $2$ mod
$24$.
\end{remark}

\section{Machine certification}\label{sec:certification}

The results are supported by machine at two independent levels: a Lean
formalization of the full Main Theorem, and exact CP certificates for
\cref{conj:global}.

\paragraph{Lean formalization (all $n$).}
The Main Theorem --- Theorems A and B, for every $n \ge 2$ --- is formalized
and kernel-checked in Lean~4 \cite{demoura2021} with Mathlib
\cite{mathlib2020}: the statement \texttt{nmin\_eq} asserts that
$2^n - 2^{\lfloor \log_2 n \rfloor}$ is the least element of
$\{N \ge 2 : A \text{ valid mod } N\}$, with no unproven assumptions (the only
axioms are Mathlib's standard \texttt{propext}, \texttt{Classical.choice},
\texttt{Quot.sound}). The development mirrors the paper proof: the greedy digit
sum is defined by a one-bit-peeling recursion, \cref{lem:step} is
\texttt{gmin\_step}, and the induction on $j$ is \texttt{slack} --- with, as in
\cref{sec:validity}, no range restriction on $j$ and no small-$n$ cases.
Theorem~B constructs the four witnesses of \cref{sec:optimality} directly.
The full development is available at
\url{https://github.com/jarfo/min-modulus}.

\paragraph{Evidence for \cref{conj:global}.}
Validity of a pair $(A, N)$ is decidable exactly --- without enumerating the
$\binom{2n-1}{n}$ multisets --- by proving a small integer feasibility model
infeasible ($n$ bounded integer variables $k_i$ with $\sum k_i = n$,
$\sum k_i a_i \equiv p \pmod N$, and $k \ne \mathbf{1}$). Leaving the set $A$
free as well (its $n$ residues become variables), CP-SAT \cite{ortools}
certifies, for each modulus $N' < 2^n - 2^m$, that \emph{no}
size-$n$ set is valid mod $N'$ --- one infeasibility certificate per modulus.
This certifies \cref{conj:global} outright for $n \le 7$; the number of moduli
roughly doubles per increment of $n$ and the certificates harden, so beyond
that the conjecture rests on search: for $n \le 13$, extensive solver- and
GPU-search finds valid sets at $N = 2^n - 2^m$ and none below, and every
minimal-modulus solution found canonicalizes, under translation and unit
scaling, to the super-increasing set.

\section{Open problems}\label{sec:open}

\begin{enumerate}
\item \emph{Global optimality} (\cref{conj:global}): prove that no size-$n$
residue set is valid below $2^n - 2^{\lfloor \log_2 n \rfloor}$. The case
analysis of \cref{sec:optimality} uses the binary structure of the fixed set
throughout; a set-free argument would need a different mechanism. Even an
exponential lower bound $\Nmin \ge c^n$ valid for all sets appears to be open.
\item \emph{Structure of minimal sets}: for $n \le 13$ every minimal-modulus
valid set found canonicalizes to the super-increasing one. Is the minimal-
modulus solution unique up to translation and unit scaling for all $n$?
\item \emph{Weighted and rectangular variants}: the validity condition fixes
$\sum_i k_i = n$ with all rows sharing one exponent set. Permanents of
rectangular matrices, or mixed row supports, lead to variants of the validity
condition whose minimal moduli are unexplored.
\item \emph{Beyond cyclic groups}: validity mod $N$ is the case $G = \Z_N$ of a
property of finite abelian groups. Say that $n$ elements $g_0, \dots, g_{n-1}$
of a finite abelian group $G$, with sum $t = \sum_j g_j$, have \emph{unique
multiset sums} if the all-ones multiset is the only size-$n$ multiset drawn from
them that sums to $t$; the companion paper \cite{fonollosa2026transform}
realizes every such family as an exact permanent evaluator by a transform of
length $|G|$, the cyclic case $G = \Z_N$ being validity mod $N$. The Main
Theorem gives the least modulus among cyclic groups, but over \emph{all} finite
abelian groups the least order is strictly smaller: the elementary abelian group
$(\Z_2)^{n-1}$, of order $2^{n-1} < 2^{\,n} - 2^{\lfloor \log_2 n \rfloor}$ for
$n \ge 3$, already carries a size-$n$ family with unique multiset sums
\cite{fonollosa2026transform}, so no cyclic group is optimal. Whether $2^{n-1}$
is the least order over \emph{all} finite abelian groups for every $n$ --- it is
for $n \le 6$ by exhaustive search --- is open; as a first step,
\cref{prop:twogroup} settles it within the elementary abelian $2$-groups, the
family underlying Glynn's scheme. In general the least order is now known to
within a factor $O(\sqrt n\,)$: any such group has order at least
$\binom{2n}{n}/2^{\,n} \sim 2^{\,n}\!/\sqrt{\pi n}$, a bound obtained in
\cite{fonollosa2026transform} by reading the permanent evaluator through the
partial-derivative method of Nisan and Wigderson \cite{nisan1997} --- a
circuit-complexity lower bound yielding a combinatorial one. The remaining
question is the exact value.
\end{enumerate}

\medskip
\noindent For the elementary abelian $2$-groups Problem~4 has an elementary
answer: $2^{n-1}$ is optimal for every $n$, by a rank count over $\F_2$.

\begin{proposition}[optimality among elementary abelian $2$-groups]\label{prop:twogroup}
If $g_0, \dots, g_{n-1} \in (\Z_2)^k$ have unique multiset sums, then $k \ge n-1$.
Hence among elementary abelian $2$-groups the least order admitting a size-$n$
embedding with unique multiset sums is exactly $2^{n-1}$, attained by
$(\Z_2)^{n-1}$ with $g = (0, e_1, \dots, e_{n-1})$.
\end{proposition}

\begin{proof}
Write $G = \F_2^k$ and let $\Lambda : \F_2^n \to \F_2^k$ be the $\F_2$-linear map
$\Lambda(x) = \sum_j x_j g_j$, so that $t = \sum_j g_j = \Lambda(\mathbf 1)$.

Since $2g = 0$ in $G$, the sum $\sum_j k_j g_j$ depends only on the parities
$\bar k_j = k_j \bmod 2$, and $\sum_j k_j g_j = t$ holds iff
$u := \bar k + \mathbf 1 \in \ker \Lambda$ (over $\F_2$, $-1 = 1$); note
$\operatorname{supp} u = \{\, j : k_j \text{ even} \,\}$. We claim the embedding has
unique multiset sums iff $\ker\Lambda$ contains no nonzero vector of even Hamming
weight. If some $u \in \ker\Lambda$ is nonzero with $\operatorname{wt}(u) = 2m$,
form $k$ by setting $k_j = 1$ off $\operatorname{supp} u$ and, on the $2m$
coordinates of $\operatorname{supp} u$, setting $m$ of them to $2$ and the other
$m$ to $0$: then $\bar k + \mathbf 1 = u$, so $\sum_j k_j g_j = t$, while
$\sum_j k_j = (n - 2m) + 2m = n$ and $k \ne \mathbf 1$ --- a genuine rival.
Conversely a rival $k \ne \mathbf 1$ yields $u = \bar k + \mathbf 1 \in \ker\Lambda$
with $u \ne 0$ (else every $k_j$ is odd, hence, being positive with sum $n$, equal
to $1$) and $\operatorname{wt}(u) = n - \operatorname{wt}(\bar k) \equiv 0 \pmod 2$,
since $\sum_j k_j = n$ forces $\operatorname{wt}(\bar k) \equiv n$. This proves the
claim.

The even-weight vectors form the parity hyperplane
$H = \{\, x \in \F_2^n : \operatorname{wt}(x) \equiv 0 \,\}$, of codimension $1$, so
$\ker\Lambda \cap H = \{0\}$ forces $\dim\ker\Lambda \le 1$; then
$\operatorname{rank}\Lambda = n - \dim\ker\Lambda \ge n-1$, and
$\operatorname{im}\Lambda \subseteq \F_2^k$ gives $k \ge n-1$, i.e.\
$|G| = 2^k \ge 2^{n-1}$. Equality holds for $(\Z_2)^{n-1}$ with
$g = (0, e_1, \dots, e_{n-1})$: there $\Lambda(x) = (x_1, \dots, x_{n-1})$ has
kernel $\langle e_0 \rangle$, spanned by a weight-one vector, so the sums are unique
at order $2^{n-1}$.
\end{proof}

\bibliographystyle{plain}
\bibliography{refs}

\end{document}